\documentclass[11pt]{amsart}

\usepackage[english]{babel}
\usepackage{amsmath,amsfonts}
\usepackage{amsthm}
\usepackage{amssymb}
\usepackage{amsrefs}
\usepackage{url}
\usepackage[all]{xy}
\usepackage{graphicx}
\usepackage{psfrag}
\usepackage{array}
\usepackage{enumerate}
\usepackage{mathtools}
\usepackage{color} %
\usepackage{dsfont} 
\usepackage{mathrsfs}
\usepackage{lipsum}
\usepackage{relsize} 
\usepackage{upgreek} 
\usepackage{caption}

\usepackage{pgf,tikz}
\usetikzlibrary{arrows}

\usepackage{xcolor}
\definecolor{rouge}{rgb}{0.85,0.1,.4}
\definecolor{bleu}{rgb}{0.1,0.2,0.9}
\definecolor{violet}{rgb}{0.7,0,0.8}
\definecolor{noir}{rgb}{0,0,0}
\usepackage[colorlinks=true,linkcolor=bleu,urlcolor=noir,citecolor=rouge]{hyperref} 

\newcommand{\norm}[1]{\Vert#1\Vert} %
\newcommand{\abs}[1]{\vert#1\vert} %
\newcommand{\dpl}{\partial}
\def\noqed{\renewcommand{\qedsymbol}{}} 

\newcommand{\bra}{\langle}
\newcommand{\ket}{\rangle}

\DeclareMathOperator{\Ldroit}{L}

\newcommand{\Poisson}[1]{\mathcal{P}\!#1}

\newcommand{\indicator}{\mathds{1}}
\newcommand{\indic}{\indicator}
\newcommand{\OrEdges}{\textrm{\upshape E}} 
\newcommand{\Vertices}{X}

\DeclareMathOperator{\cardinal}{card}
\newcommand{\card}[1]{\cardinal(#1)} %

\DeclareMathOperator{\cC}{C} 
\newcommand{\cM}{\mathcal{M}} 
\newcommand{\B}{\bar{B}} 
\DeclareMathOperator{\dEdroit}{d}
\newcommand{\ddroit}{\dEdroit\!}
\DeclareMathOperator{\Aut}{Aut}
\DeclareMathOperator{\Isom}{Isom}
\newcommand{\St}{\textbf{St}}

\DeclareMathOperator{\SL}{SL}
\DeclareMathOperator{\homology}{H}
\renewcommand{\H}{\homology}
\DeclareMathOperator{\vol}{vol}

\newcommand{\C }{\mathbf{C}} 

\newcommand{\N }{\mathbf{N}}

\newcommand{\Q }{\mathbf{Q}} %

\newcommand{\R }{\mathbf{R}} %

\newcommand{\Z}{\mathbf{Z}}%

\theoremstyle{definition}
\newtheorem{theorem}{Theorem}
\newtheorem{question}[theorem]{Question}
\newtheorem{proposition}[theorem]{Proposition}
\newtheorem{corollary}[theorem]{Corollary}
\newtheorem{lemma}[theorem]{Lemma}

\theoremstyle{definition}
\newtheorem{definition}[theorem]{Definition}
\theoremstyle{definition}
\newtheorem{remark}[theorem]{Remark}

\begin{document}
\title{Norm growth for the Busemann cocycle}

\author{Thibaut Dumont}
\address{Department of Mathematics, University of Utah\\ Salt Lake City, 84112-0090, UT, USA}
\email{thibaut.dumont@math.ch}

\begin{abstract} Using explicit methods, we provide an upper bound to the norm of the Busemann cocycle of a locally finite regular tree $X$, emphasizing the symmetries of the cocycle. The latter takes value into a submodule of square summable functions on the edges of $X$, which corresponds the Steinberg representation for rank one groups acting on their Bruhat-Tits tree. The norm of the Busemann cocycle is asymptotically linear with respect to square root of the distance between any two vertices. Independently, Gournay and Jolissaint \cite{GournayJolissaint} proved an exact formula for harmonic 1-cocycles covering the present case.
\end{abstract}


\maketitle
\setcounter{tocdepth}{1} 
{ 
\hypersetup{linkcolor=black}
\tableofcontents
}


\section{Introduction}
We present a method from the author's doctoral dissertation \cite{DumontThese}. Therein, the study of the norm growth of the cocycles introduced by Klingler in \cite{KlinglerVolume} is transported, in rank one, to the geometric Question \ref{question1} below.

Let $q\geq 2$ be an integer and $X$ be a $q+1$ regular tree with vertex set also denoted $\Vertices$ by abuse, edge set $\OrEdges$, and visual boundary $\dpl X$. The Busemann cocycle $B\colon X^2\to \cC(\dpl X)$ is given by the usual formula
\begin{equation}\label{eq:Busemann}
B(x,y)(\xi)=\lim_{z\to \xi} d(y,z)-d(x,z),
\end{equation}
where $d$ is the metric on $X$ giving length 1 each edge.
Let $\B$ denote the composition with the quotient map $\cC(\dpl X)\to \cC(\dpl X)/\C$ modding out constant functions. 

In a context of Bruhat-Tits building \cite{KlinglerPoisson}, Klingler introduces a transform named after Poisson which, in the present setting, is a map $\Poisson{\,}\colon\cC(\dpl X)\to \C^\OrEdges$ defined by integration against an $\Aut(X)$-equivariant field $\nu\colon\OrEdges\to\cM(\dpl X)$ of signed measures on~$\dpl X$:
\begin{equation}\label{eq:3}
\Poisson{\phi}(e)\coloneqq\int_{\dpl X} \phi \ddroit\nu_e,
\end{equation}
where $\phi\in\cC(\dpl X)$ and $e\in\OrEdges$, (see Section~\ref{subsection:VisualEdge} for precise definitions).
The Poisson transform is $\Aut(X)$-equivariant, factors through $\cC(\dpl X)/\C$, and maps locally constant functions
into $\ell^2(\OrEdges)$ (Proposition \ref{prop:6}). This was first proved by Klingler for Bruhat-Tits trees \cite{KlinglerPoisson}, where a similar statement is proved for more general Bruhat-Tits buildings in relation to the Steinberg representation as explained in the motivations below.

\begin{question}\label{question1} Find an upper bound for the norm $\norm{\Poisson{\B(x,y)}}_{\ell^2(\OrEdges)}$ depending only on $x,y$ (and $q$).
\end{question}
The present paper exposes the solution developed in \cite{DumontThese}*{Chap.\ 4} which emphasis the symmetries of the cocycle and hopes to inspire an approach to the higher rank case started in \cite{DumontThese}*{Chap.\ 2--3}.

For a fixed $q$, the norm depends only on $d(x,y)$ thanks to the equivariance of the construction, and one may ask to determine the asymptotic growth type of $\norm{\Poisson{\B(x,y)}}$ as $d(x,y)\to\infty$. The difficulty lies into the search of an upper bound. In fact, we prove:

\begin{theorem} \label{MainThm} For every integer $q\geq 2$, there are constants $C,K>0$ such that 
\begin{equation*}
4\,d(x,y)\leq\norm{\Poisson{\B(x,y)}}^2\leq Cd(x,y)+K, 
\end{equation*}
for all $x,y\in X$, with constants given by: 
\begin{equation*}
C=\frac{8(q+1)^2}{(q-1)^2}\quad\text{ and }\quad K=\frac{16q^2(2q+1)}{(q-1)^3(q+1)}.
\end{equation*}
\end{theorem}

In an independent work, Gournay and Jolissaint obtained a formula for the norm of harmonic cocycles \cite{GournayJolissaint}*{Theorem 1.2} which subsume our estimate. Indeed, since the average of $B(x,y)$ over the neighbors $y$ of $x$ is proportional to the indicator function on $\dpl X$, the cocycle $\B:X^2\to\cC(\dpl X)/\C$ is harmonic\footnote{This is distinct from the fact that $\sum_{o(e)=x}\nu_e=0$, which implies that $\Poisson$ ranges in the space of harmonic functions on $\OrEdges$, see Section \ref{subsection:VisualEdge}.}, {\it  i.e.\ }satisfies $\sum_{y\sim x} \B(x,y)=0$. Therefore $\Poisson{\B}$ yields an harmonic 1-cocycle of $\Aut(X)$ for its regular representation into $\ell^2(\OrEdges)$. Viewing it as an inhomogeneous 1-cocycle,  their result implies
\begin{theorem}[Gournay-Jolissaint \cite{GournayJolissaint}*{Theorem 1.2}] For every integer $q\geq 2$, there are constants $C',K'>0$ such that 
\begin{equation*}
\norm{\Poisson{\B(x,y)}}^2 = C'd(x,y)-K'(1-q^{-d(x,y)}),
\end{equation*}
for all $x,y\in X$.
\end{theorem}
The discrete Laplacian plays a central role in establishing the above formula as it is invertible in regular trees.

\subsection{On the proof of Theorem \ref{MainThm}} \label{Ontheproof} Let $e\in\OrEdges$ be an oriented edge of $X$ and let $\Aut(X)_e$ denote its stabilizer. The measure $\nu_e$ is defined on the partition $\dpl X=\Omega_e^+\sqcup\Omega_e^-$ into $\Aut(X)_e$-orbits as the difference $\nu^+_e-\nu^-_e$ of the probability measures supported on $\Omega_e^+$ and $\Omega_e^-$ respectively and proportional to a fixed visual measure. For definiteness we take $\Omega^-_e$ to be the shadow of $t(e)$, the target of $e$, casted by light emanating from $o(e)$, the origin of $e$ as in Figure \ref{fig:edge}. One should think of $e$ as being the neck of an hourglass with sand flowing from $\Omega^+_e$ through $e$ in the direction of $\Omega_e^-$.

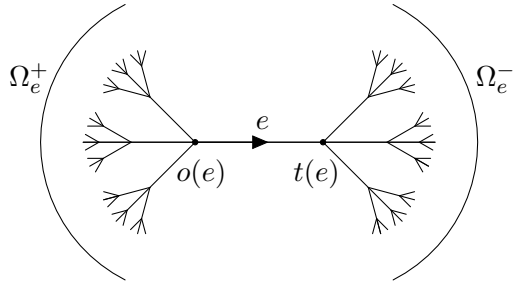
\begin{figure}[htbp]
\centering
\begin{tikzpicture}[line cap=round,line join=round,>=triangle 45,x=.850cm,y=.850cm]
\clip(-1,-2.2) rectangle (7,2.2);
\draw [->] (2.,0.) -- (3.15,0.);
\draw (1.2928932188134525,0.7071067811865475)-- (2.,0.);
\draw (1.2928932188134525,-0.7071067811865475)-- (2.,0.);
\draw (1.2928932188134525,0.7071067811865475)-- (0.9359996966841342,1.0640003033158656);
\draw (0.9359996966841342,1.0640003033158656)-- (0.7625631329235418,1.237436867076458);
\draw (1.2928932188134525,-0.7071067811865475)-- (0.9359996966841342,-1.0640003033158656);
\draw (0.9359996966841342,-1.0640003033158656)-- (0.7625631329235418,-1.237436867076458);
\draw (0.25,0.)-- (5.753703705166133,0.);
\draw (0.8053676011387001,0.8377388767319816)-- (1.2928932188134525,0.7071067811865475);
\draw (1.1622611232680182,1.1946323988613)-- (1.2928932188134525,0.7071067811865475);
\draw (0.562896489139237,0.252361829659192)-- (1.,0.);
\draw (0.562896489139237,-0.252361829659192)-- (1.,0.);
\draw (0.8053676011387001,-0.8377388767319816)-- (1.2928932188134525,-0.7071067811865475);
\draw (1.1622611232680182,-1.1946323988613)-- (1.2928932188134525,-0.7071067811865475);
\draw (0.350480947161671,0.375)-- (0.562896489139237,0.252361829659192);
\draw (0.350480947161671,-0.375)-- (0.562896489139237,-0.252361829659192);
\draw (0.5684488490966514,0.901221065013438)-- (0.8053676011387001,0.8377388767319816);
\draw (1.0987789349865618,1.4315511509033485)-- (1.1622611232680182,1.1946323988613);
\draw (0.5684488490966514,-0.901221065013438)-- (0.8053676011387001,-0.8377388767319816);
\draw (1.0987789349865618,-1.4315511509033485)-- (1.1622611232680182,-1.1946323988613);
\draw (0.9888245595074259,1.368068962621892)-- (1.1622611232680182,1.1946323988613);
\draw (1.2257433115494745,1.4315511509033485)-- (1.1622611232680182,1.1946323988613);
\draw (0.9359996966841342,1.0640003033158656)-- (0.6990809446420855,1.127482491597322);
\draw (0.9359996966841342,1.0640003033158656)-- (0.8725175084026778,1.3009190553579142);
\draw (0.5684488490966513,0.7742566884505253)-- (0.8053676011387001,0.8377388767319816);
\draw (0.6319310373781077,1.011175440492574)-- (0.8053676011387001,0.8377388767319816);
\draw (0.31762014845762104,0.252361829659192)-- (0.562896489139237,0.252361829659192);
\draw (0.44025831879842897,0.4647773716367579)-- (0.562896489139237,0.252361829659192);
\draw (1.2257433115494745,-1.4315511509033485)-- (1.1622611232680182,-1.1946323988613);
\draw (0.9888245595074259,-1.368068962621892)-- (1.1622611232680182,-1.1946323988613);
\draw (0.9359996966841342,-1.0640003033158656)-- (0.8725175084026778,-1.3009190553579142);
\draw (0.9359996966841342,-1.0640003033158656)-- (0.6990809446420855,-1.127482491597322);
\draw (0.6319310373781077,-1.011175440492574)-- (0.8053676011387001,-0.8377388767319816);
\draw (0.5684488490966514,-0.901221065013438)-- (0.8053676011387001,-0.8377388767319816);
\draw (0.5684488490966513,-0.7742566884505253)-- (0.8053676011387001,-0.8377388767319816);
\draw (0.44025831879842897,-0.4647773716367579)-- (0.562896489139237,-0.252361829659192);
\draw (0.31762014845762104,-0.252361829659192)-- (0.562896489139237,-0.252361829659192);
\draw (0.28286079870404995,0.12263817034080794)-- (0.4952763406816159,0.);
\draw (0.28286079870404995,-0.12263817034080794)-- (0.4952763406816159,0.);
\draw [shift={(2.,0.)}] plot[domain=2.0408564767205206:4.242328830459066,variable=\t]({1.*2.4182795974314706*cos(\t r)+0.*2.4182795974314706*sin(\t r)},{0.*2.4182795974314706*cos(\t r)+1.*2.4182795974314706*sin(\t r)});
\draw [shift={(4.,0.)}] plot[domain=2.0408564767205206:4.242328830459066,variable=\t]({-1.*2.4182795974314706*cos(\t r)+0.*2.4182795974314706*sin(\t r)},{0.*2.4182795974314706*cos(\t r)+1.*2.4182795974314706*sin(\t r)});
\draw (4.774256688450526,1.4315511509033483)-- (4.837738876731983,1.1946323988612997);
\draw (4.90122106501344,1.4315511509033483)-- (4.837738876731983,1.1946323988612997);
\draw (5.0111754404925755,1.3680689626218918)-- (4.837738876731983,1.1946323988612997);
\draw (5.064000303315867,1.0640003033158654)-- (5.127482491597323,1.300919055357914);
\draw (5.064000303315867,1.0640003033158654)-- (5.237436867076459,1.2374368670764577);
\draw (5.064000303315867,1.0640003033158654)-- (5.300919055357916,1.1274824915973218);
\draw (5.3680689626218925,1.0111754404925737)-- (5.194632398861301,0.8377388767319813);
\draw (5.431551150903349,0.9012210650134377)-- (5.194632398861301,0.8377388767319813);
\draw (5.431551150903349,0.774256688450525)-- (5.194632398861301,0.8377388767319813);
\draw (5.194632398861301,0.8377388767319813)-- (4.707106781186548,0.7071067811865472);
\draw (4.707106781186548,0.7071067811865472)-- (5.064000303315867,1.0640003033158654);
\draw (4.837738876731983,1.1946323988612997)-- (4.707106781186548,0.7071067811865472);
\draw (4.707106781186548,0.7071067811865472)-- (4.,0.);
\draw (4.707106781186548,-0.7071067811865477)-- (4.,0.);
\draw (5.437103510860764,0.25236182965919174)-- (5.,0.);
\draw (5.75,0.)-- (0.2462962948338676,0.);
\draw (5.437103510860764,-0.2523618296591923)-- (5.,0.);
\draw (5.194632398861301,-0.8377388767319818)-- (4.707106781186548,-0.7071067811865477);
\draw (4.707106781186548,-0.7071067811865477)-- (5.064000303315867,-1.0640003033158658);
\draw (4.837738876731983,-1.1946323988613001)-- (4.707106781186548,-0.7071067811865477);
\draw (5.559741681201572,0.4647773716367576)-- (5.437103510860764,0.25236182965919174);
\draw (5.64951905283833,0.375)-- (5.437103510860764,0.25236182965919174);
\draw (5.68237985154238,0.2523618296591917)-- (5.437103510860764,0.25236182965919174);
\draw (5.717139201295951,0.12263817034080761)-- (5.504723659318385,0.);
\draw (5.717139201295951,-0.12263817034080828)-- (5.504723659318385,0.);
\draw (5.68237985154238,-0.25236182965919235)-- (5.437103510860764,-0.2523618296591923);
\draw (5.64951905283833,-0.375)-- (5.437103510860764,-0.2523618296591923);
\draw (5.64951905283833,-0.375)-- (5.437103510860764,-0.2523618296591923);
\draw (5.559741681201572,-0.46477737163675825)-- (5.437103510860764,-0.2523618296591923);
\draw (5.431551150903349,-0.7742566884505256)-- (5.194632398861301,-0.8377388767319818);
\draw (5.431551150903349,-0.9012210650134383)-- (5.194632398861301,-0.8377388767319818);
\draw (5.3680689626218925,-1.0111754404925741)-- (5.194632398861301,-0.8377388767319818);
\draw (5.064000303315867,-1.0640003033158658)-- (5.300919055357916,-1.1274824915973223);
\draw (5.064000303315867,-1.0640003033158658)-- (5.237436867076459,-1.2374368670764582);
\draw (5.064000303315867,-1.0640003033158658)-- (5.127482491597323,-1.3009190553579144);
\draw (5.0111754404925755,-1.3680689626218923)-- (4.837738876731983,-1.1946323988613001);
\draw (4.901221065013439,-1.4315511509033487)-- (4.837738876731983,-1.1946323988613001);
\draw (4.901221065013439,-1.4315511509033487)-- (4.837738876731983,-1.1946323988613001);
\draw (4.774256688450526,-1.4315511509033487)-- (4.837738876731983,-1.1946323988613001);
\draw (-0.6,1) node {$\Omega_e^+$};
\draw (2.788083959001203,0.5646948578081253) node[anchor=north west] {$e$};
\draw (3.9,-0.5) node {$t(e)$};
\draw [fill=black] (4,0) circle (1pt);
\draw (2.1,-0.5) node {$o(e)$};
\draw [fill=black] (2,0) circle (1pt);
\draw (6.7,1) node {$\Omega^-_e$};
\end{tikzpicture}
\caption{\label{fig:edge} An oriented edge $e$.}
\end{figure}

By construction the map $e\mapsto \nu_e$ is equivariant under any automorphism of $X$ and if $\bar{e}$ denote the reverse orientation then $\nu_{\bar{e}}=-\nu_e$. Thus the Poisson transform satisfies $\Poisson{\phi}(\bar{e})=-\Poisson{\phi}(e)$. 
 Each geometric edge has therefore a preferred orientation $e$ for which $\Poisson{\B(x,y)}(e)$ is non-negative; Figure \ref{fig:preferred} illustrates this globally. In that case, the subset of $\dpl X$ where $B(x,y)$ takes its maximum, namely $d(x,y)$, has $\nu_e^+$-measure contributing more to the value of $\Poisson{\B(x,y)}(e)$ than its $\nu_e^-$-measure. Symmetrically, the set where $B(x,y)$ is minimal and equal to $-d(x,y)$ will have larger $\nu_e^-$-measure. For the latter signs cancel which explains in principle the positivity of the poisson transform for this preferred orientation. One notices the central role of the barycenter of $[x,y]$ in the symmetry.

\begin{figure}[htbp]
\centering
\begin{tikzpicture}[line cap=round,line join=round,>=triangle 45,x=1.0cm,y=1.0cm]
\foreach \x in {3}
\clip(-0.5,-0.58) rectangle (6.5,2.25);
\draw (0.5,0.)-- (1.,0.);
\draw (1.,0.)-- (2.,0.);
\draw (2.,0.)-- (4.,0.);
\draw (4.,0.)-- (5.,0.);
\draw (5.,0.)-- (5.5,0.);
\draw (0.11697777844051105,0.32139380484326974)-- (0.5,0.);
\draw (0.11697777844051105,-0.32139380484326974)-- (0.5,0.);
\draw (0.6579798566743313,0.9396926207859084)-- (1.,0.);
\draw (2.,1.5)-- (2.,0.);
\draw (0.0314727426090939,0.5563169600397468)-- (0.11697777844051105,0.32139380484326974);
\draw (-0.12922415981254093,0.3648058492600024)-- (0.11697777844051105,0.32139380484326974);
\draw (0.2749576351148424,1.2610864256291783)-- (0.6579798566743313,0.9396926207859084);
\draw (0.7448039455077965,1.4320964972920125)-- (0.6579798566743313,0.9396926207859084);
\draw (1.625,2.149519052838329)-- (2.,1.5);
\draw (2.375,2.149519052838329)-- (2.,1.5);
\draw (-0.12922415981254093,-0.3648058492600024)-- (0.11697777844051105,-0.32139380484326974);
\draw (0.0314727426090939,-0.5563169600397468)-- (0.11697777844051105,-0.32139380484326974);
\draw (4.,1.5)-- (4.,0.);
\draw (3.625,2.149519052838329)-- (4.,1.5);
\draw (4.375,2.149519052838329)-- (4.,1.5);
\draw (5.255196054492203,1.4320964972920123)-- (5.342020143325668,0.9396926207859081);
\draw (5.725042364885158,1.2610864256291778)-- (5.342020143325668,0.9396926207859081);
\draw (5.342020143325668,0.9396926207859081)-- (5.,0.);
\draw (5.883022221559489,0.3213938048432694)-- (5.5,0.);
\draw (5.883022221559489,-0.32139380484327007)-- (5.5,0.);
\draw (5.968527257390907,0.5563169600397465)-- (5.883022221559489,0.3213938048432694);
\draw (6.129224159812541,0.36480584926000204)-- (5.883022221559489,0.3213938048432694);
\draw (6.129224159812541,-0.3648058492600028)-- (5.883022221559489,-0.32139380484327007);
\draw (5.968527257390907,-0.5563169600397472)-- (5.883022221559489,-0.32139380484327007);
\draw [->] (0.6579798566743313,0.9396926207859084) -- (0.8289899283371657,0.4698463103929542);
\draw [->] (2.,1.5) -- (2.,0.75);

\draw [->] (0.11697777844051105,0.32139380484326974) --  (0.38293062669080097,0.10094781715856682); 
\draw [->] (0.11697777844051105,-0.32139380484326974) -- (0.3886698980749221,-0.08896166122344093) ; 
\draw [->] (0.2749576351148424,1.26108) -- (0.56,1.020);
\draw [->] (0.7448039455077965,1.43209) -- (0.675,1.039);
\draw [->] (1.625,2.149519052838329) -- (1.88,1.70); 
\draw [->] (2.375,2.149519052838329) -- (2.12,1.70); 
\draw [->] (4.,1.5) -- (3.73,1.97);
\draw [->] (4.,1.5) -- (4.27,1.97);
\draw [->] (4.,0.) -- (4.,0.975); 
\draw [->] (5.,0.) -- (5.237,0.65);
\draw [->] (5.5,0.) -- (5.815,0.265);
\draw [->] (5.5,0.) -- (5.815,-0.265);
\draw [->] (5.342020143325668,0.9396926207859081) -- (5.267,1.37);
\draw [->] (5.342020143325668,0.9396926207859081) -- (5.68,1.222);
\draw [->] (0.5,0.) -- (0.85,0.);
\draw [->] (1.,0.) -- (1.5,0.);
\draw [->] (2.,0.) -- (3.1,0.);
\draw [->] (4.,0.) -- (4.68,0.);
\draw [->] (5.,0.) -- (5.39,0.);
\draw (0.5,0.3) node {$x$};
\draw [fill=black] (0.5,0) circle (1.5pt);
\draw (5.5,0.35) node {$y$};
\draw [fill=black] (5.5,0) circle (1.5pt);
\end{tikzpicture}
\caption{\label{fig:preferred} The preferred orientations.}
\end{figure}
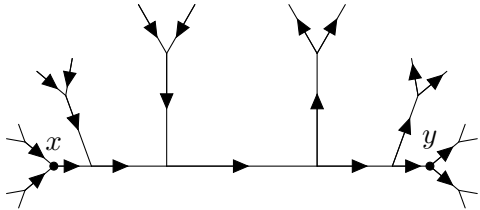

The next step derives an integration formula (Propositions \ref{prop:SumCase1} and \ref{prop:SumCase2}) for
\begin{equation*}
\Poisson{\B(x,y)}(e)=\int_{\dpl X} B(x,y)\ddroit\nu_e,
\end{equation*}
which exhibits a symmetry around the barycenter of $[x,y]$. More precisely every edge $e$ has an associated edge $e'$, which is aligned with $[x,y]$ if and only if $e$ is so, whose position relatively to $[x,y]$ is opposite to the barycenter, and for which the integration formula shows $\Poisson{\B(x,y)}(e)=\Poisson{\B(x,y)}(e')$, see Corollary \ref{cor:symmbarycenter}. Using this symmetry, one averages the aforementioned formulae for $\Poisson{\B(x,y)}(e)$ and $\Poisson{\B(x,y)}(e')$ to obtain a spherical rearrangement of the terms into non-negative quantities (Sections \ref{subsection:P} and \ref{subsection:Pj}) and in turn yields the relevant upper bounds for $\Poisson{\B(x,y)}(e)$, see equations \eqref{eq:parai} and \eqref{eq:paraij}. 

The upper bound of Theorem \ref{MainThm} is obtained by observing that the sum of $\Poisson{\B(x,y)}(e)^2$ over the edges not on $[x,y]$ is bounded independently of $x,y$ and that each edge on $[x,y]$ contributes a uniformly bounded quantity to the $\ell^2$-norm, see Section \ref{subsection:sum}.

The spherical rearrangement of Section \ref{subsection:P} also provides the lower bound by applying Cauchy-Schwarz to the indicator function of the edges of $[x,y]$ pointing towards $y$.

\subsection{Motivations}
Let $G$ be a group with a length function $L$, typically the word length function of a compact generating set of a locally compact group, and let $V$ be a Banach space endowed with a linear isometric action of $G$. In this setting, cohomology theories obtained by imposing a growth condition (with respect to $L$) on the norm of cocycles have been extensively studied in recent decades and have proven themselves powerful refinements of group cohomology, {\it e.g.\ }bounded cohomology \cite{Gromov}, \cite{MonodInvitation}, or polynomially bounded cohomology \cite{Ogle}. The polynomial theory has notable applications to the $\ell^1$-analogue of the Novikov conjecture \cite{OgleNovikov}.

Our result represents the first step of a program initiated in \cite{DumontThese} aiming at determining the norm growth of the natural cocycles introduced by Klingler~\cite{KlinglerVolume} for some (almost-simple) algebraic groups over a non-Archimedean local field of characteristic~$0$ and their \emph{Steinberg representation}. The latter, introduced by Matsumoto \cite{Matsumoto} and Shalika \cite{Shalika}, is often called the \emph{special representation} as in \cite{BorelSerre}, \cite{Borel}. 

For simplicity, we focus on the special linear group $G=\SL_n(\Q_p)$ over a $p$-adic field, $n\geq 2$. Among the (admissible) irreducible pre-unitary representation of~$G$, the Steinberg representation $\St$ is the only non-trivial coefficient module for which the cohomology of $G$ does not vanish identically in positive degree. This result was implicit in the work of Borel~\cite{BorelLaplacien} for large $p$ using ideas of Garland \cite{Garland}, but was first proved in full generality by Casselman~\cite{CasselmanGarland}. More precisely, the only non-vanishing positive degree is the rank $r=n-1$ of $G$ for which $\H^{r}(G,\St)=\C$. In \cite{KlinglerVolume}, Klingler geometrically builds an $r$-cocycle $\vol_G$ whose class spans the $r$-cohomology \cite{KlinglerVolume}*{Th\'eor\`eme 1}. Later Monod \cite{MonodInvitation}*{Problem P} suggested $\vol_G$ should be a good candidate to look for polynomial growth phenomenon. He is actually interested in `quasifying' $\vol_G$ to obtain a new polynomially bounded cohomology class of degree $r+1$ which remains an open question.

To study the norm of the cocycle $\vol_G$, one needs to understand the $G$-invariant inner product of the Steinberg representation. The Poisson transform \cite{KlinglerPoisson} relative to an Iwahori subgroup $B$ of $G$ yields an explicit isomorphism of (admissible) representations between $\St$ and a submodule $\mathcal{H}\subset\Ldroit^2(G/B)$ of (smooth) harmonic functions. Borel--Serre \cite{BorelSerre}*{\S 5.10} had earlier established this isomorphism using abstract arguments leaving no hope for explicit computations; the same holds for the argument in \cite{Borel}*{\S 6}.

Both $\vol_G$ and the Poisson transform have geometric description using the Bruhat-Tits building of $G$. Following Klingler \cite{KlinglerVolume} \cite{KlinglerPoisson}, the cocycle $\vol_G$ for $G=\SL_2(\Q_p)$ corresponds to the Busemann cocycle $\B$ of the Bruhat-Tits tree of $G$, whereas the Poisson transform relative to an Iwahori subgroup (edge stabilizer) is the map $\Poisson{\,}$ studied here. More precisely, fixing a base vertex $x$, the map $\vol_G(g,g')\coloneqq \B(gx,g'x)$ defines Klingler's homogeneous 1-cocycle for $\Aut(X)$ valued into the representation $\cC^{\infty}(\dpl X)/\C$ of locally constant functions modulo constants. 
The latter is pre-unitary because it is isomorphic, under the Poisson transform, to a submodule of $\ell^2(\OrEdges)$, and identifies to $\St$ when restricted to $G$.
This geometric description allows to somewhat ignore the role of $G$, and the definitions extend to arbitrary regular trees. More generally if $X$ is a regular locally finite Euclidean building, {\it e.g.\ }as studied in \cite{DumontThese} for type $\widetilde{A}_2$, similar constructions can and will be considered in future work.

\subsection{Acknowledgements} The author is indebted to many people and organizations which supported the present research, and wishes to thank them warmly. Mainly performed at the EPF Lausanne, the project was funded by the Swiss Confederation and the ERC Advance Grant 267635 attributed to Nicolas Monod. The author's stay at the University of Utah, which we thank for its generous hospitality, lead to the completion of the present work under the sponsorship of the Swiss NSF via an Early Postdoc.Mobility grant (Project 171852, \textit{Cohomology and Bruhat-Tits Buildings}). 

We thank also Maxime Gheysens for discussions, notably leading to the proof of the lower bound in Theorem \ref{MainThm}. 
Finally, we benefited greatly from the supervision of Nicolas Monod who suggested this research topic; we are profoundly grateful. 

\section{Preliminaries}
We start with some preliminaries regarding the Poisson transform and the Busemann cocycle. Our conventions for graphs and trees follow Serre's book \cite{Serre}. 

Let $q\geq 2$ be an integer and $X$ a $(q+1)$-regular tree. 
In addition $X$ comes with an extremity map $(o,t)\colon\OrEdges\to\Vertices\times\Vertices$, assigning to each edge $e$ its origin $o(e)$ and target $t(e)$, and an orientation reversing map $e\mapsto\bar{e}$. A \emph{geometric edge} is understood to be a set of the form $\{e,\bar{e}\}$.

We identify the tree $X$ with its geometric realization and endow it with the metric for which geometric edges have length $1$. We denote $S_R(x)$ the sphere about $x$ of radius $R\in[0,\infty)$. 
The visual boundary of $X$, denoted $\dpl X$, is the set of asymptotic classes of geodesic rays $r\colon[0,\infty)\to X$. Endowed with the cone topology, $\dpl X$ is compact metrizable and totally disconnected; the basis for the topology being given by the following family of closed-open subsets: 
\begin{equation}\label{eq:4}
\Omega_x(z)\coloneqq\{\xi\in\dpl X\mid z \text{ sit on the geodesic ray from }x \text{ to }\xi\},
\end{equation}
with $x,z\in X$. The set $\Omega_x(z)$ is called the \emph{shadow} of $z$ from $x$ (or from light emanating from $x$). The visual boundary $\dpl X$ is also homeomorphic to the projective limit of the system
\begin{equation*}
\{ S_{n+1}(x)\to S_{n}(x)\mid n\in\N\},
\end{equation*}
for any base point $x\in \Vertices$, where each sphere is given the discrete topology.

\subsection{Busemann Cocycle}\label{subsection:Busemann}
For every pair of vertices $(x,y)\in \Vertices^2$, the function $z\mapsto d(y,z)-d(x,z)$ can be extended to the visual boundary via \eqref{eq:Busemann}. 
The induced map $B(x,y)\colon\dpl X\to \R$ is continuous and called the \emph{Busemann cocycle}, a terminology justified by the 1-cocycle identity:
\begin{equation*}
B(x,z)=B(x,y)+B(y,z).
\end{equation*}
The Busemann cocycle is a locally constant function on $\dpl X$ as we now see by identifying its values and level sets. 
We consider two fixed vertices $x$ and $y$. Given $\xi\in\dpl X$, let $r$ be the unique geodesic ray from $x$ to $\xi$. By definition, the value of Busemann cocycle at $\xi\in \dpl X$ is given by:
\begin{equation*}
B(x,y)(\xi)=\lim_{t\to\infty} d(y,r(t))-d(x,r(t))=\lim_{t\to\infty} d(y,r(t))-t.
\end{equation*}
The argument of the limit is in fact constant as soon as $r(t)$ reaches the geodesic ray $r'$ joining $y$ to $\xi$. More precisely, 
\begin{equation*}
B(x,y)(\xi)=t'-t
\end{equation*}
for all $t,t'\geq 0$ which satisfy $r'(t')=r(t)$. Set $d\coloneqq d(x,y)$ and $k\coloneqq d(x,r')$, then
\begin{equation*}
B(x,y)(\xi)=d-2k.
\end{equation*}
Consider the geodesic segment $\sigma\colon[0,d]\to X$ from $x$ to $y$; the level set for the above value is given by:
\begin{equation*}
B(x,y)^{-1}(d-2k)=\{\xi \mid B(x,y)(\xi)=d-2k\}=\Omega_x(\sigma(k))\smallsetminus\Omega_x(\sigma(k+1)),
\end{equation*}
for integers $0\leq k< d$, and equals $\Omega_x(y)$ otherwise.

\subsection{Visual measure $\nu_x$ and Radon-Nikodym derivative}\label{subsection:VisualVertex}
The well-known probability measures $\nu_x$ on $\dpl X$ were, to the best of our knowledge, introduced by Cartier in \cite{Cartier} and \cite{CartierBourbaki}*{\S 8} as tree analogues of the visual measures on the visual boundary of the hyperbolic plane. We refer also to the book of Fig\`a-Talamanca and Nebbia \cite{FigaTalamancaNebbia}*{Chap~II} for visual measures.
Fix a base vertex~$x$, the Borel probability measure $\nu_x$, called \emph{visual measure at $x$}, can be defined as the projective limit of the uniform probability measures on the spheres $S_n(x)$, with $n\in\N$. In other words, it is the unique measure on $\dpl X$ satisfying 
\begin{equation*}
\nu_x(\Omega_x(z))={\card{S_n(x)}}^{-1}=\frac{1}{q^{n-1}(q+1)}, 
\end{equation*}
for all $z\in S_n(x)$ and all $n\in \N^*$. Different base points $x,y\in \Vertices$ yield absolutely continuous measures for which the Radon--Nikodym derivative is given by: 
\begin{equation*}
\frac{\ddroit\nu_x}{\ddroit\nu_y}(\xi)=q^{B(x,y)(\xi)},
\end{equation*}
for all $\xi\in\dpl X$. By construction, $x\mapsto\nu_x$ is equivariant under automorphisms of~$X$.

\subsection{Visual measure $\nu_e$ and Poisson transform}\label{subsection:VisualEdge} We now detail the construction of the signed measure $\nu_e$ on $\dpl X$ associated to an oriented edge $e\in\OrEdges$. We merely translate into a geometric language the \emph{Poisson transform relative to Iwahori subgroups} and its associated measures developed by Klingler in \cite{KlinglerPoisson}.

The field of measures $e\mapsto\nu_e$ constructed below is naturally $\Aut(X)$-equivariant and satisfies $\nu_{\bar{e}}=-\nu_e$. First, we consider $\nu_e^+$, the Borel probability measure supported on $\Omega_e^+\coloneqq\Omega_{t(e)}(o(e))$ (see~\eqref{eq:4} for the notation) obtained from $\nu_{o(e)}$ by restriction and scaling as follows: 
\begin{equation*}
\frac{q}{q+1}\cdot\nu_e^+=\nu_{o(e)}|_{\Omega_e^+}.
\end{equation*}
On the complement of $\Omega_e^+$, we define $\nu_e^-$ to be the Borel probability measure on $\Omega_e^-\coloneqq\Omega_{o(e)}(t(e))$ given by: 
\begin{equation*}
\frac{1}{q+1}\cdot\nu_e^-=\nu_{o(e)}|_{\Omega_e^-}.
\end{equation*}
Finally, we define the signed measure associated to $e$ to be $\nu_e\coloneqq \nu_e^+-\nu_e^-$. The \emph{Poisson transform} $\Poisson{\phi}\colon\OrEdges\to\C$ of a continuous function $\phi\colon\dpl X\to \C$ is defined by integration, see equation \eqref{eq:3}.

\begin{remark} If $\phi$ is constant then $\Poisson{\phi}=0$, hence the Poisson transform of a function depends only on its class modulo constant functions. 
\end{remark}

The Poisson transform ranges in the space of alternate \emph{harmonic} functions on $\OrEdges$. For if $x\in \Vertices$, we have
\begin{equation*}
\sum_{o(e)=x} \nu_e=0,\quad \text{ 
thus
}
\quad
\sum_{o(e)=x} \Poisson{\phi}(e)=0.
\end{equation*}
We reproduce and adapt the argument from \cite{DumontThese}*{Proposition 2.3.6} showing that locally constant functions have square summable Poisson transform.
\begin{proposition} \label{prop:6}Let $\phi$ be a locally constant function on $\dpl X$, then $\Poisson{\phi}$ is in $\ell^2(\OrEdges)$.
\end{proposition}
\begin{proof} By compactness, $\phi$ takes only finitely many values and consequently there is a finite disjoint open cover $\{\Omega_x(u)\mid u\in S_R(x)\}$ into small shadows, $R$-distant from a base point $x$, on each piece of which $\phi$ is constant. Let $K$ be the pointwise stabilizer of the ball $B_R(x)$. Then $\phi$ is $K$-invariant and so is $\Poisson{\phi}$. Each tree $T_u$ rooted $u\in S_R(x)$ in the complement of $B_R(x)$ has visual boundary $\Omega_x(u)$ and is acted upon by $K$ with quotient a ray. Thus $\Poisson{\phi}$ is constant on each level of $T_u$ provided all edges point toward $u$. Using harmonicity, the sum of $\Poisson{\phi}(e)$ over $q$ adjacent edges of $T_u$ on level $j\geq1$ equals $q$ times the value of their common ancestor edge of level $j-1$, again provided the edges point toward $u$. Then a geometric series argument on each ray shows that
\begin{equation*} 
\norm{\Poisson{\phi}}_{\ell^2(\OrEdges)}\leq C\max_{d(e,x)\leq R+1}\abs{\Poisson{\phi(e)}},
\end{equation*}
where $C$ depends on $R$ and $q$ only.
\end{proof}

\section{The proof}
This section aims at proving the following estimate of the $\ell^2(\OrEdges)$-norm of the Poisson transform of $B(x,y)$:
\begin{equation}\label{eq:ineq}
\sum_{e\in\OrEdges} \Poisson{B(x,y)}(e)^2\leq C d(x,y)+C',
\end{equation}
where $C,C'$ are the constants of Theorem \ref{MainThm} depending only on $q$. The lower bound of the theorem is proved at the end of Section \ref{subsection:P}

Here is a detailed outline of the proof.

Firstly, in Section \ref{subsection:parametrization}, the oriented edges are parametrized according to their position relatively to $x$ and~$y$. We distinguish edges aligned with the geodesic segment $[x,y]$ from those not, the projection of the latter onto $[x,y]$ is well defined and equals neither $x$ nor~$y$. The distances of their extremities to $x$ and $y$ are implicitly used as parameters.

Secondly, we obtain the integration formulae for $\Poisson{B(x,y)}(e)$ (Propositions \ref{prop:SumCase1} and \ref{prop:SumCase2}) depending only on the aforementioned parameters (and $q$). The strategy is to decompose the boundary $\dpl X$ into countably many disjoint open sets whose $\nu_e$-measure is easily computed, on each of which $B(x,y)$ is constant, and then apply countable additivity of the measure. In general, removing a finite subtree of $X$ yields a decomposition of $\dpl X$ into finitely many disjoint open subsets, the boundaries of the connected components of the resulting forest. For instance, the level sets of $B(x,y)$ are obtained by removing the segment $[x,y]$, see Section \ref{subsection:Busemann}. On the other hand, Section \ref{subsection:VisualEdge} defines $\nu_e$ by removing the geometric edge $\{e,\bar{e}\}$ and looking at the resulting partition $\dpl X=\Omega_e^+\sqcup\Omega_e^-$. Suppose there is a geodesic $\sigma$ containing $e$ and $[x,y]$, the aligned case, one could naively remove their convex hull and obtain a finite open partition of $\dpl X$. However the resulting finite sum representing $\Poisson{B(x,y)}(e)$ has no obvious closed form. Instead, picking $\sigma\colon\R\to X$ to be a geodesic line, with $\sigma(0)=x$ say, we obtain a countable partition of $\dpl X$ indexed over $\Z$ into open sets with explicit $\nu_e$-measure (Lemma \ref{lem:nue1}) and on which $B(x,y)$ is constant and given by the piecewise affine function $f$ defined in \eqref{eq:f}. The case where $e$ is not aligned with $[x,y]$ requires more work, but follows similar principles, see Section~\ref{subsection:IntegrationFormulae}.

Thirdly, from the integration formula emerges a symmetry described in the introductory Section \ref{Ontheproof}. Averaging the integration formulae of two symmetric edges provides a rearrangement into non-negative terms spherically distributed around a point. The averaging technique depends again on whether $e$ is aligned with $[x,y]$ or not (Section \ref{subsection:P} and Section \ref{subsection:Pj} respectively). 
This is achieved in the last part of the paper.

Finally, we gather in Section \ref{subsection:sum} the above ingredients to complete the proof of inequality~\eqref{eq:ineq}.

\subsection{Parametrization of the edges}\label{subsection:parametrization}

 We fix $x,y\in\Vertices$ and an edge $e\in\OrEdges$. Without loss of generality we may assume $d\coloneqq d(x,y)$ is at least $2$. The possible configurations of $x$, $y$, $e$ in $X$ fall into two cases depending implicitly only on the distances between $x$, $y$ and the extremities of $e$:
\begin{enumerate}[(A)]
\item \label{itm:A} $e$ is aligned with $[x,y]$,
\item \label{itm:B} $e$ is not aligned with $[x,y]$.
\end{enumerate}
In \eqref{itm:A}, we shall always assume that 
there is a geodesic line $\sigma\colon\R\to X$ such that $x=\sigma(0)$, $y=\sigma(d)$ and $(o(e),t(e))=(\sigma(i),\sigma(i+1))$ for some $i\in\Z$. This corresponds to the preferred orientation of $e$ discussed in Section \ref{Ontheproof}. 

On the other hand, we choose for \eqref{itm:B} to orientate $e$ toward $[x,y]$, hence we assume the existence of a geodesic segment $\tau\colon[0,j]\to X$ such that $(o(e),t(e))=(\tau(0),\tau(1))$ and for which $\tau(j)$ is the projection of $e$ onto $[x,y]$. This orientation does not necessarily give positiveness of the Poisson transform of $B(x,y)$, but it provides a uniform treatment of~\eqref{itm:B}.  In this case $\tau(j)\neq x,y$ and one may extend $\tau$ to a geodesic $\tau\colon(-\infty,j]\to X$. When $q\geq 3$, one can extend $\tau$ further into a geodesic line $\tau\colon\R\to X$ intersecting $[x,y]$ only at $\tau(j)$, forming a cross with $[x,y]$. However if $q=2$, it is not possible to do so; this case gets special attention in Remark \ref{rem:q2}. We therefore assume $q\geq 3$. Like in case \eqref{itm:A}, we fix a geodesic $\sigma\colon\R\to X$ with $x=\sigma(0)$, $y=\sigma(d)$, so that the projection of $e$ onto $[x,y]$ is $\sigma(i)=\tau(j)$, for some integer $1< i < d$.
\begin{definition} We say that an oriented edge $e$ has \emph{parameter} $i\in \Z$ if it is described as in case \eqref{itm:A} above with $(o(e),t(e))=(\sigma(i),\sigma(i+1))$. Otherwise, $e$ is said to have \emph{parameters} $(i,j)$, with $1<i<d$ and $j\geq 1$, as in case \eqref{itm:B}, and its projection onto $[x,y]$ corresponds to $\sigma(i)=\tau(j)$.
\end{definition}

\begin{lemma} \label{lem:6} The number of edges with parameter $i\in\Z$ is given by:
\begin{equation*}
n(i)=\begin{cases} q^{\abs{i}} & \text{ if }  i<0,\\
1 & \text{ if } 0\leq i < d, \\
q^{i-d} & \text{ if } d\leq i,
\end{cases}
\end{equation*}
whereas there are $n(i,j)=(q-1)q^{j-1}$ edges with parameters $(i,j)$, for all $1\leq i\leq d-1$ and $j\geq 1$.\qed
\end{lemma}

\subsection{Integration formulae}\label{subsection:IntegrationFormulae} In this section, we derive the integration formula for $\Poisson{B(x,y)}(e)$ which depends only on the parameters of $e$. With the parametrizations of the previous section, two countable partitions of $\dpl X$ are obtained by removing the geodesic lines $\sigma$ and $\tau$.

The first partition $\{\Omega^\sigma_k\mid k\in\Z\}$ is obtained by removing only $\sigma$ and is used for the aligned case \eqref{itm:A}. For every $k\in\Z$, let 
\begin{equation*}
\Omega^\sigma_k\coloneqq \bigsqcup_{\substack{z\in S_1(\sigma(k)) \\ z\neq\sigma(k\pm1)}} \Omega_{\sigma(0)}(z).
\end{equation*}
Equivalently, we remove from $X$ the open geometric edges of the geodesic line $\sigma$ and consider the visual boundary components of the resulting forest. Of course, the end points of $\sigma$ are $\nu_e$-null sets and can be ignored. In the right hand side, the base point $\sigma(0)=x$ can be replaced by any node of~$\sigma$. 

The $\nu_e$ measure of $\Omega_k^\sigma$ can be computed using the Sections \ref{subsection:VisualVertex} and \ref{subsection:VisualEdge}. Suppose $\Omega_k^\sigma\subset \Omega^+_e$, the intuition is that of one starts at the origin $o(e)=\sigma(i)$ of $e$ with a bag of sand of total mass one and then distributes this sand equally to the neighbouring nodes except $t(e)=\sigma(i+1)$, hence dividing the total sand in $q$ equal piles. Repeating this process along $\sigma$, one reaches $\sigma(k)$ with having divided by $q$ as many times as there are edges between $\sigma(i)$ and $\sigma(k)$, namely $d(\sigma(k),\sigma(i))=\abs{k-i}$. From here $1/q$ of the mass continues its travel along $\sigma$ and the remaining $\frac{q-1}{q}$ of the mass at $\sigma(k)$ will constitute the $\nu_e$-measure of $\Omega_k^\sigma$.

\begin{lemma}[Proposition 4.3.4, \cite{DumontThese}]\label{lem:nue1}
Suppose $e$ is aligned with $[x,y]$ and parametrized by $i\in\Z$, then the $\nu_e$-measure of $\Omega^\sigma_k$ is given by:
\begin{equation}\label{eq:20}
\nu_e(\Omega^\sigma_k)=\frac{(q-1)}{q}\begin{cases} q^{-\abs{k-i}} & \text{ if } k-i\leq 0,\\ -q^{-(k-i)+1} & \text{ if } k-i >0,\end{cases}
\end{equation}
for all $k\in\Z$. \qed
\end{lemma}
To synthetize the right hand side, we define the following continuous real functions $g(x)=q^{-\abs{x}}$ (Figure \ref{fig:graphg}) and 
\begin{equation*}
g_{\frac{1}{2}}(x)=
\begin{cases}
g(x) & \text{ if } x\leq 0,\\
1-2x & \text{ if } 0\leq x\leq 1,\\
-g(x-1)& \text{ if } 1 \leq x,
\end{cases}
\end{equation*}
so that \eqref{eq:20} becomes $\nu_e(\Omega^\sigma_k)=\frac{(q-1)}{q} g_{\frac{1}{2}}(k-i)$, for all $k\in\Z$. 
The presence of the index~$\frac{1}{2}$ emphasizes the central symmetry of the graph of $g_{\frac{1}{2}}$ about $(\frac{1}{2},0)\in\R^2$ (Figure \ref{fig:graphg12}). 

\begin{proposition}[Integration formula for parameter $i$]\label{prop:SumCase1} Let $e$ be an edge parametrized by $i\in\Z$. Then the Poisson transform of $B(x,y)$ evaluated at $e$ is given by:
\begin{equation*}
\Poisson{B(x,y)}(e)=\frac{(q-1)}{q}\sum_{k\in\Z}f(k) g_{\frac{1}{2}}(k-i),
\end{equation*}
where $f$ is the continuous piecewise affine function (of Figure \ref{fig:graphf}) given by:
\begin{equation}\label{eq:f}
f(x)=
\begin{cases}
d&  \text{ if } x \leq 0,\\
d-2x &\text{ if } 0\leq x \leq d,\\
-d & \text{ if } d\leq x.
\end{cases}
\end{equation}
\end{proposition}

\begin{proof} By construction $B(x,y)$ is constant on $\Omega^\sigma_k$ where it takes value $f(k)$ thanks to Section \ref{subsection:Busemann}. We can apply countable additivity to integrate $B(x,y)$ against $\nu_e$, 
\begin{align*}
\Poisson{B(x,y)}(e)&=\int_{\dpl X} B(x,y)\ddroit\nu_e= \sum_{k\in\Z} \int_{\Omega^\sigma_k} B(x,y)\ddroit\nu_e
=\sum_{k\in\Z} f(k)\nu_e(\Omega^\sigma_k)\\&=\frac{(q-1)}{q}\sum_{k\in\Z} f(k)g_{\frac{1}{2}}(k-i).\qedhere
\end{align*}
\end{proof}

Consider now an edge $e$ with parameters $(i,j)$. The adequate partition is obtained similarly removing further the geodesic $\tau$ containing $e$ and intersecting with the sets $\Omega^\sigma_k$. More precisely, for every $l\in\Z$, let 
\begin{equation*}
\Omega^\tau_l\coloneqq \bigsqcup_{\substack{z\in S_1(\tau(l)) \\ z\neq\tau(l\pm1)}} \Omega_{\tau(0)}(z)
\end{equation*}
and $\Omega_{k,l}\coloneqq\Omega^\sigma_{k}\cap\Omega^\tau_l,$ for all $k\in\Z$. 

The tree $X$ consists of the cross formed by $\sigma$ and $\tau$ with rooted trees attached at the nodes of $\sigma$ or $\tau$. Note $\Omega_{k,l}$ is empty whenever $l\neq j$ and $k\neq i$. When $l=j$, the set $\Omega_{k,j}$ is the boundary of the tree attached to $\sigma(k)$ and for $k=i$ the set $\Omega_{i,l}$ is the boundary of the tree attached to $\tau(i)$. This is provided $(k,l)\neq(i,j)$; the branching at $\sigma(i)=\tau(j)$ has various configurations depending on the valency of $X$. For instance if $l>j$, the set $\Omega_{i,l}$ is non-empty if and only if $\tau$ can indeed be extended to form a cross with $\sigma$, that is if $q\geq 3$.
Finally the center node of the cross also have a tree attached with boundary $\Omega_{i,j}$ which is empty if and only if $q=3$. This is fortunately covered by the last formula of Lemma \ref{lem:nue2}.

Intuitively, the mass spreads from $e$ along $\tau$ following $g_{\frac{1}{2}}$ similarly to the previous case. At the node $\sigma(i)=\tau(j)$, the mass entering $\sigma$ is $2/q\cdot g_\frac{1}{2}(j)$ and spreads uniformly along $\sigma$ in both directions according to the function $h:\R\to\R$ of Figure  \ref{fig:graphh}. It is given by:
\begin{equation*}
h(x)=\begin{cases} g(x+1)& \text{if }x\leq -1,\\
1 &\text{if } -1\leq x\leq 1\text{ and } x\neq 0,\\
0& \text{if } x=0,\\
g(x-1) &\text{if }x\geq 1,
\end{cases}
\end{equation*}
with a discontinuity at $x=0$ introduced for later convenience.

\begin{lemma}[Proposition 4.4.6, \cite{DumontThese}]\label{lem:nue2}
The $\nu_e$-measure of $\Omega_{k,l}$ is given by:
\begin{align*}
\nu_e(\Omega_{k,l})&=0 && \text{for all }k\neq i\text{ and }  l\neq j ,\\
\nu_e(\Omega_{i,l})&= \frac{(q-1)}{q}\cdot g_{\frac{1}{2}}(l)\quad\big(=\nu_e(\Omega^\tau_l)\,\big), && \text{for all }l\neq j, \\
\nu_e(\Omega_{k,j})&=\frac{(q-1)}{q}\cdot g_{\frac{1}{2}}(j)\cdot \frac{1}{q}\cdot h(k-i), && \text{for all } k\neq i,\\
\nu_e(\Omega_{i,j})&=\frac{(q-3)}{q}\cdot g_\frac{1}{2}(j).
\end{align*}
\end{lemma}

\begin{proposition}[Integration formula for parameter $(i,j)$]\label{prop:SumCase2} Let $e$ be an edge with parameters $(i,j)$, then $\Poisson{B(x,y)}(e)$ is given by: 
$$
\Poisson{B(x,y)}(e)=-\frac{2}{q}f(i)g_\frac{1}{2}(j)+\frac{(q-1)}{q^2}g_\frac{1}{2}(j)\left(\sum_{k\in\Z} f(k)h(k-i)\right) .
$$
\end{proposition}
\begin{proof} We apply countable additivity to the partition $\{\Omega_{k,l}\mid k,l\in\Z\}$.
\begin{align*}
\Poisson{B(x,y)}(e)&=\int_{\dpl X} B(x,y)\ddroit\nu_e= \sum_{k,l\in\Z} \int_{\Omega_{l,k}} B(x,y)\ddroit\nu_e
=\sum_{k,l\in\Z} f(k)\nu_e(\Omega_{l,k})\\
&= f(i)\nu_e(\Omega_{i,j})+\sum_{k\neq i}f(k)\nu_e(\Omega_{k,j})+f(i)\sum_{l\neq j}\nu_e(\Omega_{i,l})\\
&=\frac{(q-3)}{q}f(i)g_{\frac{1}{2}}(j)+\frac{(q-1)}{q^2}g_{\frac{1}{2}}(j)\sum_{k\in\Z}f(k)h(k-i)+\frac{(q-1)}{q}f(i)\sum_{l\neq j}g_{\frac{1}{2}}(l).
\end{align*}
One concludes using $\sum_{l\neq j}g_{\frac{1}{2}}(l)=-g_\frac{1}{2}(l)$.
\end{proof}

\begin{remark}\label{rem:q2} When $q=2$, the set $\Omega_{j,i}$ is empty and the ray $\tau\colon(-\infty,j]\to X$ cannot be extended further to form a cross with $\sigma$. In that case the above becomes
\begin{align*}
\Poisson{B(x,y)}(e)&=\sum_{k\neq i}f(k)\nu_e(\Omega_{k,j})+f(i)\sum_{l<j}\nu_e(\Omega_{i,l})=\sum_{k\neq i}f(k)\nu_e(\Omega_{k,j})-f(i)\sum_{l\geq j}\nu_e(\Omega_{l,i})
\end{align*}
This integration formula equals that of Proposition \ref{prop:SumCase2} provided $\sum_{l\geq j}\nu_e(\Omega_{i,l})=g_\frac{1}{2}(j)$, which holds thanks to $\sum_{l\geq j}g_\frac{1}{2}(l)=q\cdot g_\frac{1}{2}(j)$ (and $q=2$).
\end{remark}

Let $P(i)$ denote the value of the right hand side in Proposition \ref{prop:SumCase1} and $P(i,j)$ that of Proposition \ref{prop:SumCase2}. We use the notation $\bra\cdot,\cdot \ket$ for the standard pairing of $\ell^p(\Z)$ and its dual $\ell^q(\Z)$ and $\uptau$ the regular representation of $\Z$ on $\ell^p(\Z)$. The integration formulae for $P(i)$ and $P(i,j)$ can be written as:
\begin{equation}\label{eq:IntegrationFormulae}
P(i)=\frac{(q-1)}{q}\bra f,\uptau_i g_\frac{1}{2}\ket\quad\text{and}\quad P(i,j)=-\frac{2}{q}f(i)g_\frac{1}{2}(j)+\frac{(q-1)}{q^2}g_\frac{1}{2}(j)\bra f,\uptau_ih\ket.
\end{equation}

\subsection{Summation}\label{subsection:sum}
 In the last section, we obtain the following bounds. On the one hand, for an edge $e$ parametrized by~$i\in\Z$, 
\begin{equation}\label{eq:parai}
\abs{\Poisson{B(x,y)}(e)}=\abs{P(i)}\leq
\begin{cases} 
\frac{2}{(q-1)}q^{-(\abs{i}-1)} & \text{ if }  i<0,\\
\frac{2(q+1)}{(q-1)} & \text{ if } 0\leq i < d, \\
\frac{2}{(q-1)}q^{-(i-d-1)} & \text{ if } d\leq i.
\end{cases}
\end{equation}
On the other hand, if $e$ is an edge parametrized by~$(i,j)$, with~$j\geq1$, then
\begin{equation}\label{eq:paraij}
\abs{\Poisson{B(x,y)}(e)}=\abs{P(i,j)}\leq
\begin{cases} 
\frac{2}{(q-1)}q^{-(i+j-1)} & \text{ if }  1\leq i\leq d/2,\\
\frac{2}{(q-1)}q^{-(d-i+j-1)} & \text{ if } d/2\leq i\leq d-1.
\end{cases}
\end{equation}
We now use them to prove \eqref{eq:ineq} hence the upper bound of Theorem \ref{MainThm}.
\begin{proof}[Proof of the upper bound of Theorem \ref{MainThm}. Case \eqref{itm:A}] The number $n(i)$ of edges aligned with $[x,y]$ with parameter $i$ is obtained in Lemma \ref{lem:6}. It only accounts for the preferred orientations, hence the factor $1/2$ below.
\begin{align*} 
\frac{1}{2}\sum_{\substack{e\text{ aligned} \\ \text{with } [x,y]}}\Poisson{B(x,y)}(e)^2&=\sum_{i\in\Z} n(i) P(i)^2\\
&\leq  \left(\frac{2(q+1)}{(q-1)}\right)^2 \cdot d(x,y) + 2\left(\frac{2}{(q-1)}\right)^2\sum_{i> 0}q^iq^{-2(i-1)}\\
&= \frac{4(q+1)^2}{(q-1)^2} \cdot d(x,y) + \frac{8q^2}{(q-1)^3}.\qedhere
\end{align*} \noqed
\end{proof}
\begin{proof}[Case \eqref{itm:B}] The graph of the right hand side of \eqref{eq:paraij} is symmetric with respect to the axis $y=d/2$, meaning it is invariant under $i\mapsto d-i$. In fact we show in Section \ref{subsection:Pj} that $i\mapsto P(i,j)$ satisfies the same of symmetry. This allows us to only sum over $1\leq i\leq d/2$. Similarly to the previous proof:
\begin{align*}
\frac{1}{2}\sum_{\substack{e\text{ not aligned} \\ \text{with } [x,y]}}\Poisson{B(x,y)}(e)^2&=\sum_{\substack{1\leq i\leq d-1\\j\geq1}}n(i,j)P(i,j)^2 \leq 2\sum_{\substack{1\leq i\leq d/2\\j\geq1}}n(i,j)P(i,j)^2\\
&= 2\sum_{\substack{1\leq i\leq d/2\\j\geq1}}(q-1)q^{j-1}P(i,j)^2\\
&\leq 2(q-1) \sum_{\substack{1\leq i\leq d/2\\j\geq1}}q^{j-1}\left(\frac{2}{(q-1)}q^{-(i+j-1)}\right)^2\\
&= \frac{8}{(q-1)}\sum_{\substack{1\leq i\leq d/2\\j\geq0}}q^{-2i}q^{-j}\\
&\leq \frac{8}{(q-1)} \sum_{j\geq0}q^{-j}\sum_{i\geq0}q^{-2i}
=\frac{8q^3}{(q-1)^3(q+1)}.\qedhere
\end{align*}
\end{proof}

\section{Symmetry and spherical rearrangement}
This section presents the averaging methods which, applied to the integration formulae \eqref{eq:IntegrationFormulae}, yield the estimates \eqref{eq:parai} and \eqref{eq:paraij}. The series $\bra f, \uptau_i g_\frac{1}{2}\ket$ and $\bra f, \uptau_ih\ket$ in \eqref{eq:IntegrationFormulae} are manipulated to obtain a spherical rearrangement of their terms.

Recall that $\sigma$ is a geodesic with $x=\sigma(0)$ and $y=\sigma(d)$, so that the barycenter of $[x,y]$ corresponds to $\sigma(d/2)$. As mentioned in the introduction, the Poisson transform of $B(x,y)$ is symmetric about $\sigma(d/2)$. For edges $e$ of type $(i,j)$, the map $i\mapsto P(i,j)$ indeed shows a symmetry (with a sign) around $d/2$, notably because $\sigma(i)$ is the projection of $e$ onto $[x,y]$. On the other hand, the parameter $i$ of an edge $e$ aligned with $[x,y]$ corresponds to its origin $o(e)=\sigma(i)$. Therefore the symmetry around the barycenter translates into $i\mapsto P(i)$ being symmetric about $(d-1)/2$, see Proposition \ref{prop:symmetries}.

To concretize the above discussion, we introduce some notations. For a real valued function $f\colon\R\to\R$ (or $f\colon\Z\to\R$), we write $\uptau_tf(x)=f(x-t)$, $\check{f}(x)=(f)\check{\,}(x)=f(-x)$ and $(\check{\uptau}_tf)(x)=f(x+t)=\uptau_{-t}f(x)$. The operators $\uptau_t$ and $\check{\,}$ correspond to the action of $t$ and $-1$ for the canonical linear action of $\Isom(\R)=\R\rtimes \{\pm1\}$ (resp.~$\Z\rtimes \{\pm 1\}$) onto the space of functions over $\R$ (resp. over~$\Z$). 
Denote further $\bra f_1,f_2\ket=\sum_{k\in\Z}f_1(k)f_2(k)$ when the series is well defined, {\it e.g.\ }absolutely convergent.
We shall use the following identities freely:
\begin{equation*}
\bra \uptau_t f_1,\uptau_t f_2\ket= \bra f_1, f_2\ket=\bra \check{f_1},\check{f_2}\ket,\quad
\check{\check{f}}=f,\quad \text{and}\quad
(\uptau_tf)\check{\,}=\uptau_{-t}\check{f}=\check{\uptau}_t\check{f}.
\end{equation*}

\begin{definition}\label{def:symmetry} A function $f\colon\R\to\R$ (or $f\colon\Z\to\R$) is said to have a \emph{(central) symmetry about $h\in\R$} (resp.\ $h\in\frac{1}{2}\Z$) if its graph is invariant under the central symmetry about $(h,0)\in\R^2$. Equivalently $f$ satisfies $-\check{f}=\uptau_{-2h}f$. 

We say that $f$ has an  \emph{(axial) symmetry about $y=h$} if its graph is invariant under the reflexion through the vertical line $y=h$. Equivalently $f$ satisfies $\check{f}=\uptau_{-2h}f$. 
\end{definition}
The following is clear from the graphs of Figure \ref{fig:graphf}--\ref{fig:graphh}.
\begin{lemma}\label{lem:symmetries} The functions $f$, $g$, $g_\frac{1}{2}$, and $h$ defined in Section \ref{subsection:IntegrationFormulae} satisfy
\begin{align*}
-\check{f}&=\uptau_{-d}f &&\text{(symmetry about }d/2),\\
\check{g}&=g \quad\text{and} \quad\check{h}=h &&\text{(symmetry about }y=0),\\
-\check{g}_\frac{1}{2}&=\uptau_{-1}g_\frac{1}{2} &&\text{(symmetry about }1/2).
\end{align*}
\end{lemma}

\begin{figure}[htbp]
\captionsetup{margin=10pt,font=small}
\centering

\begin{minipage}[b]{0.4\textwidth}

\begin{tikzpicture}[line cap=round,line join=round,>=triangle 45,x=.550cm,y=.550cm]
\draw[->,color=black] (-4,0.) -- (6.2,0.);
\foreach \x in {2.,4.}
\draw[shift={(\x,0)},color=black] (0pt,2pt) -- (0pt,-2pt) node[below] {};
\draw[->,color=black] (0.,-2.16) -- (0.,3.8);
\foreach \y in {-2.,0.,2.}
\draw[shift={(0,\y)},color=black] (2pt,0pt) -- (-2pt,0pt) node[left] {};
\draw[color=black] (5pt,-6pt) node {\footnotesize $0$};
\draw[color=black] (4,0.2) node[above] {$d$};
\draw[color=black] (2.1,0.2) node[above] {$\frac{d}{2}$};
\draw[color=black] (.2,2) node[right] {$d$};
\draw[color=black] (.2,-2) node[right] {$-d$};
\clip(-4,-3) rectangle (6.5,4.5);
\draw[line width=1.2pt,color=black,smooth,samples=100,domain=0:4] plot(\x,{-\x+2});
\draw[line width=1.2pt,color=black,smooth,samples=100,domain=-4:0] plot(\x,{2});
\draw[line width=1.2pt,color=black,smooth,samples=100,domain=4:6] plot(\x,{-2});
\draw[color=black] (-3.0,1.5) node {$f$};
\begin{scriptsize}
\end{scriptsize}
\end{tikzpicture}
\caption{\label{fig:graphf} Graph of $f$.}
  \end{minipage}
     \hfill
  \begin{minipage}[b]{0.4\textwidth}
\centering 
\begin{tikzpicture}[line cap=round,line join=round,>=triangle 45,x=.55cm,y=.55cm]
\draw[->,color=black] (-4.3,0.) -- (4.8,0.);
\foreach \x in {-4.,-3.,-2.,-1.,1.,2.,3.,4.}
\draw[shift={(\x,0)},color=black] (0pt,2pt) -- (0pt,-2pt) node[below] {};
\draw[->,color=black] (0.,-2.16) -- (0.,3.8);
\foreach \y in {-2.,0.,2.}
\draw[shift={(0,\y)},color=black] (2pt,0pt) -- (-2pt,0pt) node[left] {};
\draw[color=black] (5pt,-6pt) node {\footnotesize $0$};
\clip(-4.3,-2) rectangle (4.3,5.3);
\draw[line width=1.2pt,color=black,smooth,samples=100,domain=0:4.3] plot(\x,{2*exp(-\x)});
\draw[line width=1.2pt,color=black,smooth,samples=100,domain=-4.3:0] plot(\x,{2*exp(\x)});
\draw[color=black] (-4.0,0.5) node {$g$};
\begin{scriptsize}
\draw[color=black] (.0,2) node[right] {$1$};
\draw[color=black] (1.0,.1) node[above] {$1$};
\end{scriptsize}
\end{tikzpicture}
    \caption{\label{fig:graphg} Graph of $g$.}
\end{minipage}

\captionsetup{margin=10pt,font=small}
  \begin{minipage}[b]{0.4\textwidth}
\begin{tikzpicture}[line cap=round,line join=round,>=triangle 45,x=.55cm,y=.55cm]
\draw[->,color=black] (-4.3,0.) -- (4.8,0.);
\foreach \x in {-4.,-3.,-2.,-1.,1.,2.,3.,4.}
\draw[shift={(\x,0)},color=black] (0pt,2pt) -- (0pt,-2pt) node[below] {};
\draw[->,color=black] (0.,-2.16) -- (0.,3.8);
\foreach \y in {-2.,0.,2.}
\draw[shift={(0,\y)},color=black] (2pt,0pt) -- (-2pt,0pt) node[left] {};
\draw[color=black] (5pt,-6pt) node {\footnotesize $0$};
\clip(-4.3,-2) rectangle (4.3,5.3);
\draw[line width=1.2pt,color=black,smooth,samples=100,domain=0:1] plot(\x,{2-4*\x});
\draw[line width=1.2pt,color=black,smooth,samples=100,domain=1:4.3] plot(\x,{-2*exp(-\x+1)});
\draw[line width=1.2pt,color=black,smooth,samples=100,domain=-4.3:0] plot(\x,{2*exp(\x)});
\draw[color=black] (-3.8,0.5) node {$g_{\frac{1}{2}}$};
\begin{scriptsize}
\draw[color=black] (.0,2) node[right] {$1$};
\draw[color=black] (1.0,.1) node[above] {$1$};
\end{scriptsize}
\end{tikzpicture}
\caption{\label{fig:graphg12} Graph of $g_\frac{1}{2}$.}
  \end{minipage}
     \hfill
  \begin{minipage}[b]{0.4\textwidth}
 
\begin{tikzpicture}[line cap=round,line join=round,>=triangle 45,x=.550cm,y=.550cm]
\draw[->,color=black] (-4.3,0.) -- (4.8,0.);
\foreach \x in {-4.,-3.,-2.,-1.,1.,2.,3.,4.}
\draw[shift={(\x,0)},color=black] (0pt,2pt) -- (0pt,-2pt) node[below] {};
\draw[->,color=black] (0.,-1) -- (0.,3.8);
\draw[color=black] (5pt,-6pt) node {\footnotesize $0$};
\clip(-4.3,-1) rectangle (4.3,4);
\draw[line width=1.2pt,color=black,smooth,samples=100,domain=-1:-0.1] plot(\x,{2});
\draw[line width=1.2pt,color=black,smooth,samples=100,domain=0.1:1] plot(\x,{2});
\draw [fill=black] (0,0) circle (1.5pt);
\draw (0,2) circle (1.5pt);
\draw[line width=1.2pt,color=black,smooth,samples=100,domain=1:4.3] plot(\x,{2*exp(-\x+1)});
\draw[line width=1.2pt,color=black,smooth,samples=100,domain=-4.3:-1] plot(\x,{2*exp(\x+1)});
\draw[color=black] (-4.0,0.5) node {$h$};
\begin{scriptsize}
\draw[color=black] (.0,1.7) node[right] {$1$};
\draw[color=black] (1.0,.1) node[above] {$1$};
\end{scriptsize}
\end{tikzpicture}
\caption{\label{fig:graphh} Graph of $h$.}
\end{minipage} 
\end{figure}

When studying $P(i,j)$, there is no obstacle 
to work with arbitrary parameters $i,j\in \Z$. 

\begin{proposition}\label{prop:symmetries} Let $P_j(i)\coloneqq P(i,j)$, then
\begin{align}
\check{P}&=\uptau_{-d+1}P && \text{(symmetry about }y=(d-1)/2),\label{eq:symP}\\
-\check{P}_j&=\uptau_{-d}P_j && \text{(symmetry about }d/2).\label{eq:symPj}
\end{align}
\end{proposition}

\begin{proof}
Using the above and \eqref{eq:IntegrationFormulae}, the following computation
\begin{align*}
\frac{q}{q-1}\check{P}(i)&=\bra f,\uptau_{-i}g_\frac{1}{2}\ket
=\bra f,\check{\uptau}_{i}g_\frac{1}{2}\ket=\bra \check{f},\uptau_{i}\check{g}_\frac{1}{2}\ket
=\bra \uptau_{-d}f,\uptau_{i-1}g_\frac{1}{2}\ket=\bra f,\uptau_{d+i-1}g_\frac{1}{2}\ket\\
&= \frac{q}{q-1}P(d+i-1)= \frac{q}{q-1}\uptau_{-d+1}P(i),
\end{align*}
shows $\check{P}=\uptau_{-d+1}P$ which means $P$ is symmetric about $y=(d-1)/2$.
Similarly one obtains $-\check{P}_j=\uptau_{-d}P_j$. 
\end{proof}

\begin{corollary}\label{cor:symmbarycenter} Let $e$ be an edge with parameter $i\in \Z$, then $\Poisson{B(x,y)}(e)=\Poisson{B(x,y)}(e')$ for all edge $e'$ with parameters $d-i-1$. On the other hand if $e$ has parameter $(i,j)$, with $1< i<d$ and $j\geq 1$, then $\Poisson{B(x,y)}(e)=-\Poisson{B(x,y)}(e')$ for all edge $e'$ with parameters~$(d-i,j)$\qed
\end{corollary}

\subsection{Averaging $P$}\label{subsection:P} This section establishes the bound \eqref{eq:parai} using an average of $P$.
The symmetry of $P$ is encoded by equation \eqref{eq:symP} of Proposition \ref{prop:symmetries}, which can also be written $P=\uptau_{d-1}\check{P}.$ The mean of the two sides of the latter is
\begin{equation*}
P(i)=\frac{1}{2}(P(i)+P(d-i-1))=\frac{q-1}{2q}\left(\bra f,\uptau_ig_\frac{1}{2}\ket+\bra f,\uptau_{d-i-1}g_\frac{1}{2}\ket\right).
\end{equation*}
By transferring the operators $\uptau$ to the left side of the pairing, one obtains
\begin{equation*}
P(i)=\frac{q-1}{2q}\left(\bra \uptau_{-i}f,g_\frac{1}{2}\ket+\bra \uptau_{-d+i+1}f,g_\frac{1}{2}\ket\right).
\end{equation*}
Consider the linear operator $T_i\coloneqq\frac{1}{2}(\uptau_{-i}+\uptau_{-d+i+1})$; the previous equation becomes
\begin{equation}\label{eq:PiAverage}
P(i)=\frac{q-1}{q}\bra T_if,g_\frac{1}{2}\ket.
\end{equation}
This has the advantage of being computable, thanks to $f$ being piecewise affine, yielding a simple expression for $T_if$. Since $f$ has a symmetry around $d/2$, one verifies that $T_if$ has one around $1/2$. Indeed, 
\begin{align*}
(T_if)\check{\,}&=\frac{1}{2}\left((\uptau_{-i}f)\check{\,}+(\uptau_{-d+i+1}f)\check{\,}\right)=\frac{1}{2}\left(\check{\uptau}_{-i}\check{f}+\check{\uptau}_{-d+i+1}\check{f}\right)\\
&=-\frac{1}{2}\left(\uptau_{i}\uptau_{-d}f+\uptau_{d-i-1}\uptau_{-d}f\right)=-\frac{1}{2}\uptau_{-1}\left(\uptau_{-d+i+1}f+\uptau_{-i}f\right)
=-\uptau_{-1}T_if.\nonumber
\end{align*}
This symmetry is key to proving
\begin{proposition}\label{prop:Pinorm} With the above notations,
\begin{equation}\label{eq:Pinorm}
P(i)=\frac{2(q-1)}{q}\norm{T_if \cdot g_\frac{1}{2}}_{\ell^1(\N^*)}.
\end{equation}
\end{proposition}

\begin{proof} Translates of $f$ change sign only at their center of symmetry, so that the same holds for $T_if$ which in turn has the same sign as $g_\frac{1}{2}$. Precisely, $T_if(x)\geq0$ for $x\leq1/2$ and $T_if(x)<0$ otherwise. From \eqref{eq:PiAverage}, we have
\begin{equation*}
\frac{q}{q-1}P(i)=\bra T_if,g_\frac{1}{2}\ket\overset{(\text{i})}{=}\bra \abs{T_if},\abs{g_\frac{1}{2}}\ket=\norm{T_if \cdot g_\frac{1}{2}}_{\ell^1(\Z)}\overset{(\text{ii})}{=}2\norm{T_if \cdot g_\frac{1}{2}}_{\ell^1(\N^*)},
\end{equation*}
where (i) follows from $T_if$ and $g_\frac{1}{2}$ having the same sign and (ii) from the fact that the pointwise product preserves the shared symmetry about $1/2$.
\end{proof}

\begin{proposition}[Proposition 4.3.18, \cite{DumontThese}]\label{prop:tifBound} For every $k\in\N^*$, the average
\begin{equation*}
(T_if)(k)=\frac{1}{2}(f(k+i)+f(k+d-i-1))
\end{equation*}
is non-negative and bounded above by: 
\begin{equation}\label{eq:tfi}
T_if(k)\leq
\begin{cases} 
(k-\abs{i})\cdot\indic_{[-i,\infty)}(k) & \text{ if }  i<0,\\
2(k-\frac{1}{2}) & \text{ if } 0\leq i < d, \\
(k-(i-d+1))\cdot\indic_{[-d+i+1,\infty)}(k) & \text{ if } d\leq i.
\end{cases}
\end{equation}
Therefore the $\ell^1(\N^*)$-norm of the pointwise product of $T_if$ and $g_\frac{1}{2}$ is bounded by: 
\begin{equation}\label{eq:normtfi}
\norm{T_if\cdot g_\frac{1}{2}}_{\ell^1(\N^*)}\leq
\begin{cases} 
\frac{q}{(q-1)^2}q^{-(\abs{i}-1)} & \text{ if }  i<0,\\
\frac{q(q+1)}{(q-1)^2} & \text{ if } 0\leq i < d, \\
\frac{q}{(q-1)^2}q^{-(i-d)} & \text{ if } d\leq i,
\end{cases}
\end{equation}
which, applied to \eqref{eq:Pinorm}, proves \eqref{eq:parai}.
\end{proposition}
\begin{proof} For the first part \eqref{eq:tfi}, one writes $f$ as a sum of affine and constant functions each supported on an interval determined by the two cut points of $f$, namely $0$ and~$d$. Then one may compute explicitly the translates of $f$ and $T_if$ in terms of affine and indicator functions. As a result of some cancelations, $T_if$ vanishes on the interval $[1/2,-i]$ when $i<0$ and on $[1/2, -d+i+1]$ when $d\leq i$. It follows that $\abs{T_if}$ is bounded by the affine functions of \eqref{eq:tfi}. The bounds are asymptotically sharp in the sense that $\abs{T_if}$ converge pointwise to them as $d\to\infty$.
The second part \eqref{eq:normtfi} is deduced by direct computations by plugging \eqref{eq:tfi} into \eqref{eq:Pinorm}. 
\end{proof}

This spherical rearrangement for the edges of $[x,y]$ is sufficient to prove the lower bound of Theorem \ref{MainThm}.
\begin{proof}[Proof of the lower bound of Theorem \ref{MainThm}] As mentioned in the previous proof, $T_if$ does not vanish around $1/2$ when $0\leq i <d$. In fact, for such parameters $i$, one sees that $\abs{T_if(k)}\geq 1$ for all $k\in \Z$. 
Hence by \eqref{eq:Pinorm},
\begin{equation*}
\Poisson{B(x,y)}(e)=\frac{2(q-1)}{q}\norm{T_if\cdot g_\frac{1}{2}}_{\ell^1(\N^*)}\geq \frac{2(q-1)}{q}\norm{g_\frac{1}{2}}_{\ell^1(\N^*)}=2,
\end{equation*}
for all edges of $[x,y]$ pointing from $x$ to $y$. 
Applying Cauchy-Schwarz to the indicator function of the later edges and $\Poisson{B(x,y)}$ yields
\begin{equation*}
\sum_{e\subset [x,y]} \Poisson{B(x,y)}(e)\leq \sqrt{d}\norm{\Poisson{B(x,y)}}.
\end{equation*}
The left hand side is at least $2d$, hence the result.
\end{proof}

\subsection{Averaging $P_j=P(-,j)$} \label{subsection:Pj}We proceed similarly to last section in order to establish \eqref{eq:paraij}, with a twist in complexity due to the particular form of the integration formula \eqref{eq:IntegrationFormulae} for~$P(i,j)$. We may focus on  the case $1\leq i\leq d/2$ thanks to the symmetry \eqref{eq:symPj} of~$P_j$. The latter translates into $P_j=-\uptau_{d}\check{P}_j$. Therefore,
\begin{align*}
P_j(i)&=\frac{1}{2}(P_j(i)-P_j(d-i))=-\frac{2}{q}f(i)g_\frac{1}{2}(j)+\frac{(q-1)}{2q^2}g_\frac{1}{2}(j)\left(\bra f,\uptau_ih\ket-\bra f,\uptau_{d-i}h\ket\right)\\
&=-\frac{2}{q}f(i)g_\frac{1}{2}(j)+\frac{(q-1)}{2q^2}g_\frac{1}{2}(j)\bra \uptau_{-i}f-\uptau_{i-d}f,h\ket.
\end{align*}
We consider the linear operator $\tilde{T}_i\coloneqq\frac{1}{2}(\uptau_{-i}-\uptau_{i-d})$ to rewrite the above as:
\begin{equation}\label{eq:PijAverage}
P_j(i)=-\frac{2}{q}f(i)g_\frac{1}{2}(j)+\frac{(q-1)}{q^2}g_\frac{1}{2}(j)\bra \tilde{T}_if,h\ket.
\end{equation}
Here $\tilde{T}_if$ is symmetric with respect to the axis $y=0$; for one quickly verifies $(\tilde{T}_if)\check{\,}=\tilde{T}_if$.

\begin{proposition}[Section 4.4.1, \cite{DumontThese}] 
For $1\leq i\leq d/2$,  we have $\tilde{T}_if\geq0$ and
\begin{equation*}
\bra\tilde{T}_if,h\ket=2\norm{\tilde{T}_if\cdot h}_{\ell^1(\N^*)}.
\end{equation*}
Moreover
\begin{equation}\label{eq:finalPj}
P_j(i)=\frac{-2q^{-i}}{q-1}g_\frac{1}{2}(j)(1-q^{-f(i)}),
\end{equation}
Since $f(i)=d-2i\geq0$, the absolute value of \eqref{eq:finalPj} yields \eqref{eq:paraij}.
\end{proposition}

\begin{proof} We proceed as in Proposition \ref{prop:tifBound}. One may write $\tilde{T}_if=\frac{1}{2}(\uptau_{-i}f-\uptau_{i-d}f)$ explicitly as a piecewise affine map and observe its non-negativeness on its support $[i-d,d-i]$. 
Consequently, 
\begin{equation*}
\bra \tilde{T}_if,h\ket=\norm{\tilde{T}_if\cdot h}_{\ell^1(\Z)}=2\norm{\tilde{T}_if\cdot h}_{\ell^1(\N^*)},
\end{equation*}
where in the last equation we used that $h(0)=0$ and the symmetry about $y=0$. 
For the second part, we use additional notations:
\begin{equation*}
\Delta(i)\coloneqq\norm{\tilde{T}_if\cdot h}_{\ell^1(\N^*)}\quad \text{and}\quad\Sigma(i)\coloneqq\Delta(i)-\frac{q}{q-1}f(i),
\end{equation*}
so that $P_j$ can be written as:
\begin{equation}\label{eq:29}
P_j(i)=\frac{2(q-1)}{q^2}g_\frac{1}{2}(j)\left(\Delta(i)-\frac{q}{q-1}f(i)\right)=\frac{2(q-1)}{q^2}g_\frac{1}{2}(j)\Sigma(i).
\end{equation}
For $k\geq0$, the function $\tilde{T}_if$ is more explicitly described on its support by:
\begin{equation*}
\tilde{T}_if(k)=\begin{cases}
d-2i=f(i) & \text{ if }  0\leq k \leq i,\\
d-i-k=f(i)-(k-i) & \text{ if } i\leq k \leq d-i.
\end{cases}
\end{equation*}
Therefore
\begin{align*}
\Delta(i)&=\norm{\tilde{T}_if\cdot h}_{\ell^1(\N^*)}= \sum_{k=1}^{d-i}f(i)q^{-k+1}-\sum_{k=i}^{d-i}(k-i)q^{-k+1}\\
&=f(i)\sum_{k=1}^{d-i}q^{-k+1}-q^{-i}\sum_{k=0}^{f(i)}kq^{-k+1}.
\end{align*}
We set $x\coloneqq q^{-1}$ and plug the previous equation into $\Sigma(i)$:
\begin{align*}
\Sigma(i)&=\Delta(i)-\frac{1}{1-x}f(i) =f(i)\sum_{k=1}^{d-i}x^{k-1}-x^{i}\sum_{k=0}^{f(i)}kx^{k-1}- \frac{1}{1-x}f(i)\\
&=f(i)\frac{1-x^{d-i}}{1-x}-\frac{x^i}{(1-x)^2}\left(f(i)x^{f(i)+1}-(f(i)+1)x^{f(i)}+1\right)- \frac{1}{1-x}f(i)\\
&=f(i)\frac{-x^{d-i}}{1-x}-\frac{x^i}{(1-x)^2}\left(f(i)x^{f(i)+1}-(f(i)+1)x^{f(i)}+1\right)\\
&=f(i)\frac{-x^{i}x^{f(i)}}{1-x}-\frac{x^i}{(1-x)^2}\left(f(i)x^{f(i)+1}-(f(i)+1)x^{f(i)}+1\right)\\
&=\frac{-x^{i}}{(1-x)^2}\left(f(i)x^{f(i)}(1-x)+f(i)x^{f(i)+1}-(f(i)+1)x^{f(i)}+1\right)\\
&=\frac{-x^{i}}{(1-x)^2}\left(-x^{f(i)}+1\right)=\frac{-q^2q^{-i}}{(q-1)^2}\left(1-q^{-f(i)}\right).
\end{align*}
Finally \eqref{eq:29} becomes
\begin{align*}
P_j(i)&=\frac{2(q-1)}{q^2}g_\frac{1}{2}(j)\Sigma(i)=\frac{2(q-1)}{q^2}g_\frac{1}{2}(j)\frac{-q^2q^{-i}}{(q-1)^2}\left(1-q^{-f(i)}\right)\\
&=\frac{-2q^{-i}}{q-1}g_\frac{1}{2}(j)(1-q^{-f(i)}),
\end{align*}
as desired.
\end{proof}


\bibliographystyle{plain}

\bibliography{BibliographyBusemannTree}

\end{document}